\documentclass{amsart}

\DeclareMathOperator{\Ad}{Ad}
\DeclareMathOperator{\ad}{ad}

\newcommand\lie[1]{\mathfrak{#1}}
\newcommand{\g}{\lie{g}}

\newcommand{\R}{\mathbb{R}}

\newcommand{\F}{\mathcal{F}}

\newcommand{\Ass}{{\rm Ass}}
\newcommand{\TAut}{{\rm TAut}}
\newcommand{\tder}{\mathfrak{tder}}
\renewcommand{\div}{{\rm div}}
\newcommand{\Tr}{\mathfrak{tr}}
\newcommand{\tr}{{\rm tr}}
\newcommand{\ch}{{\rm ch}}
\renewcommand{\k}{\mathbb{K}}
\newcommand{\Lie}{\mathfrak{lie}}

\theoremstyle{plain}
\newtheorem{theorem}{Theorem}

\newtheorem{proposition}[theorem]{Proposition}

\theoremstyle{definition}

\newtheorem{remark}{Remark}

\begin{document}

\title[]{On triviality of the Kashiwara-Vergne problem
for quadratic Lie algebras}

\author{Anton Alekseev}
\address{Section de math\'ematiques, Universit\'e de Gen\`eve,
2-4 rue du Li\`evre, c.p. 64, 1211 Gen\`eve 4, Switzerland}
\email{alekseev@math.unige.ch}

\author{Charles Torossian}
\address{Institut Math\'ematiques de Jussieu, Universit\'e Paris 7,
CNRS, Case 7012, 2 place Jussieu, 75005 Paris, France }
\email{torossian@math.jussieu.fr}

\date{July 3, 2009}

\begin{abstract}
We show that the Kashiwara-Vergne (KV) problem for quadratic Lie algebras
(that is, Lie algebras admitting an invariant scalar product) 
reduces to the problem of representing the Campbell-Hausdorff
series in the form  $\ln(e^xe^y)=x+y+[x,a(x,y)]+[y,b(x,y)]$,
where $a(x,y)$ and $b(x,y)$ are Lie series in $x$ and $y$.
This observation explains the existence of explicit rational 
solutions of the quadratic KV problem \cite{V,AMquad}
whereas constructing an explicit rational solution of the 
full KV problem would probably require the knowledge of a rational 
Drinfeld associator. It also gives, in the case of quadratic
Lie algebras, a direct proof of the Duflo theorem 
(implied by the KV problem).
\end{abstract}


\maketitle

Let $\k$ be a field of characteristic zero, and let 
$\Lie_n$ be the degree completion of the free Lie algebra
with $n$ generators. The Campbell-Hausdorff series 
$\ch(x,y)=\ln(e^xe^y)$ is an element of $\Lie_2$.
Similarly, $\ch(x_1,\dots,x_n)=\ln(e^{x_1}\dots e^{x_n})$
is an element of $\Lie_n$. For $t \in \k^*$, let 
$\ch_t(x_1,\dots,x_n)=t^{-1}\ch(tx_1,\dots,tx_n)$.
Note that $\ch_t(x_1,\dots,x_n)$ is analytic in $t$.

The universal enveloping algebra
of $\Lie_n$, $U(\Lie_n)=\Ass_n$ is the degree completion
of the free associative algebra with $n$ generators.
For an element $\alpha \in \Ass_n$
there is a unique decomposition 
$\alpha=\alpha_0+ \sum_{i=1}^n (\partial_i \alpha) x_i$
with  $\alpha_0 \in \k$ and $\partial_i \alpha  \in \Ass_n$.
Extend the adjoint action of $\Lie_n$ to an algebra homomorphism
${\rm ad}: \Ass_n \rightarrow {\rm End}(\Lie_n)$. It is easy
to see that for $\alpha \in \Lie_n$ one has
$$
\frac{d}{ds} \, \alpha(x_1, \dots, x_i+sz,\dots, x_n)|_{s=0} =
{\rm ad}(\partial_i \alpha) z.
$$

We denote by $C_n \subset \Ass_n$ the subspace spanned by
commutators ($ab-ba \in C_n$ for all $a,b \in \Ass_n$).
Let $\tau$ be the unique anti-involution on $\Ass_n$ 
defined by the property $\tau(\alpha)=-\alpha$ for all $\alpha \in \Lie_n$.
Denote by $A_n \subset \Ass_n$ the eigenspace of $\tau$ 
corresponding to the eigenvalue $(-1)$, $\alpha \in A_n$ if
$\tau(\alpha)=-\alpha$.
Define $\Tr_n = \Ass_n/C_n$  the graded vector space of 
cyclic words in $n$ letters, and denote by 
$\tr: \Ass_n \rightarrow \Tr_n$
the natural projection. In particular, we have
$\tr(\alpha\beta)=\tr(\beta\alpha)$ for all $\alpha, \beta\in \Ass_n$.
Similarly, define $\Tr_n^{quad}=
\Ass_n/\langle A_n, C_n \rangle$ and $\tr^{quad}: \Ass_n \rightarrow
\Tr_n^{quad}$ the corresponding projection. Here $\langle A_n, C_n \rangle$
is the subspace of $\Ass_n$ spanned by $A_n$ and $C_n$. 
The definition implies $\tr^{quad}(\alpha)=\tr^{quad}(\tau(\alpha))$
for all $\alpha \in \Ass_n$ and 
$\tr^{quad}(\alpha\beta)=\tr^{quad}(\beta\alpha)$ 
for all $\alpha, \beta\in \Ass_n$. Note that for all $\alpha \in \Lie_n$
and all $k$ odd we have $\tr^{quad}(\alpha^k)=\tr^{quad}(\tau(\alpha^k))=
(-1)^k\tr^{quad}(\alpha^k)=0$.

Let $\tder_n$ be the Lie algebra of derivations of $\Lie_n$
with an extra property that for $u \in \tder_n$ there exist
$a_1, \dots, a_n \in \Lie_n$ such that $u(x_i)=[x_i, a_i]$.
Elements of $\tder_n$ are called tangential derivations. 
If we assume that $a_i$ does not contain a linear term proportional
to $x_i$, the correspondence between tangential derivations
and $n$-tuples $(a_1, \dots, a_n)$ is one-to-one.
We define simplicial maps $\tder_n \rightarrow \tder_{n+1}$.
For instance, for $u=(A,B) \in \tder_2$ we introduce
$$
\begin{array}{lll}
u^{1,2} & = & (A(x,y), B(x,y), 0), \\
u^{2,3} & = & (0, A(y,z), B(y,z)), \\
u^{12,3} & = & (A(\ch(x,y), z), A(\ch(x,y),z), B(\ch(x,y),z)), \\
u^{1,23} & = & (A(x,\ch(y,z)), B(x, \ch(y,z)),B(x,\ch(y,z)) ).
\end{array}
$$

Recall (Proposition 3.6, \cite{AT}) that 
$\div(u)=\sum_{i=1}^n \tr(x_i (\partial_i a_i))$ is a 1-cocycle on
$\tder_n$ with values in $\Tr_n$. Similarly, we define
$\div^{quad}(u)=\sum_{i=1}^n \tr^{quad}(x_i (\partial_i a_i))$.
It is a 1-cocyle on $\tder_n$ with values in $\Tr_n^{quad}$.
It is easy to see that the divergence transforms in a natural way
under simplicial maps. For example, for $u \in \tder_2$ we define
$g(x,y)=\div(u) \in \Tr_2$ and we have
$$
\begin{array}{ll}
\div(u^{1,2})=g(x,y), &
\div(u^{2,3})=g(y,z), \\
\div(u^{12,3})=g(\ch(x,y),z), &
\div(u^{1,23})=g(x, \ch(y,z)).
\end{array}
$$
For example, to prove the third equation we use that 
$$
\tr(x(\partial_x a(\ch(x,y), z))+y(\partial_y a(\ch(x,y),z)) )=
\tr(\ch(x,y) (\partial_1 a)(\ch(x,y),z)).
$$
The same transformation properties under simplicial maps 
hold for $\div^{quad}$.

Let $\TAut_n$ be the subgroup of automorphisms of $\Lie_n$ with
an extra property that for each $g\in \TAut_n$ there exist
$b_1, \dots, b_n \in \Lie_n$ such that $g(x_i)=\exp(\ad_{b_i})x_i$.
The group $\TAut_n$ is isomorphic to $\tder_n$ with group
multiplication defined by the Campbell-Hausdorff series.
Simplicial maps lift to group homomorphisms $\TAut_n \rightarrow \TAut_{n+1}$.

% The cocycles $\div$ and $\div^{quad}$ integrate to unique
% (additive) group cocycles $j: \TAut_n \rightarrow \Tr_n, 
% j^{quad}: \TAut_n \rightarrow \Tr^{quad}_n$ such that 
%
% $$
% \div(u)=\, \frac{d}{ds} \, j(\exp(su))|_{s=0} 
% \hskip 0.5cm , \hskip 0.5cm 
% \div^{quad}(u)=\, \frac{d}{ds} \, j^{quad}(\exp(su))|_{s=0}  
% \hskip 0.5cm .
% $$

The Kashiwara-Vergne (KV) problem \cite{KV} can be stated in the following
way:

\vskip 0.2cm

{\bf Kashiwara-Vergne problem}: Find a pair of Lie series
in two variables $A,B \in \Lie_2$ such that
\begin{equation} \label{kv1}
(1-\exp(-\ad_x))A(x,y) + (\exp(\ad_y)-1)B(x,y)=x+y-\ch(y,x),
\end{equation}
and
\begin{equation} \label{kv2}
\tr(x (\partial_x A) + y (\partial_y B)) =
\frac{1}{2} \, \tr( f(x) + f(y) - f(\ch(x,y)) ),
\end{equation}
where $f(x)=x/(e^x-1)-1+x/2=\sum_{k=2}^\infty B_k x^k/k!$ 
is the generating series of Bernoulli numbers.

\vskip 0.2cm

Let $\g$ be a finite dimensional Lie algebra over $\k$. Then,
the positive solution of the KV problem implies
the Duflo theorem \cite{Duf} for $\g$ (an isomorphism $Z(U\g) \cong (S\g)^\g$
between the center of the universal enveloping algebra and the ring of 
invariant polynomials) and  the cohomology isomorphism
$H(\g, U\g) \cong H(\g, S\g)$ \cite{Sh, PT}. For $\k=\R$, one also 
obtains the  extension of the Duflo isomorphism to germs of 
invariant distributions \cite{ADS}, \cite{AST}.

Let $\g$ be a finite dimensional quadratic Lie algebra over $\k$.
That is, $\g$ carries  an  invariant non-degenerate symmetric
bilinear form ({\em e.g.} $\g$ is semi-simple, but not necessarily, 
see \cite{Medina}). Then, the Duflo theorem (both algebraic and 
analytic versions) as well as  the cohomology isomorphism 
$H(\g, U\g) \cong H(\g, S\g)$ follow from  a weaker version of 
the KV problem:

\vskip 0.2cm
{\bf Quadratic Kashiwara-Vergne problem}: Find a pair of Lie series
in two variables $A,B \in \Lie_2$ which verify equation
\eqref{kv1} and
\begin{equation} \label{kvquad2}
\tr^{quad}(x (\partial_x A) + y (\partial_y B)) =
\frac{1}{2} \, \tr^{quad}(f(x)+f(y)-f(\ch(x,y)) ),
\end{equation}
with $f(x)=x/(e^x-1)-1+x/2$.
\vskip 0.2cm

\begin{remark}   \label{quad}
This reduction of the second KV equation from $\Tr_2$ to $\Tr_2^{quad}$
is related to the following property of traces in the adjoint representation
of a quadratic Lie algebra $\g$. Let $\tau_\g$ be the unique involution
of $U\g$ such that $\tau_\g(\alpha)=-\alpha$ for all $\alpha\in \g$.
Then, for all $\alpha \in U\g$ we have $\tr_\g \ad(\tau(\alpha)) =
\tr_\g \ad(\alpha)$. At the level of free Lie algebras, this property
leads to replacing $\tr$ by $\tr^{quad}$.
\end{remark}

It is obvious that equation \eqref{kv1} admits many solutions.
Indeed, let $a(x,y)$ and $b(x,y)$ be Lie series given by the following
formulas
$$
a(x,y)=\, \frac{1-\exp(-\ad_x)}{\ad_x} \, A(x,y) \hskip 0.5cm , \hskip 0.5cm
b(x,y)=\, \frac{\exp(\ad_y)-1}{\ad_y} \, B(x,y).
$$
Then, equation \eqref{kv1} takes the form
$$
[x, a(x,y)]+[y, b(x,y)]= x+y-\ch(y,x).
$$
The right hand side is given by a series in Lie monomials of 
degree greater or equal to two. 
Since each Lie monomial starts either with a Lie bracket with $x$ or with 
a Lie bracket with $y$, we obtain a solution of equation
\eqref{kv1} for each explicit presentation of the Campbell-Hausdorff
formula in terms of Lie monomials ({\em e.g.} using the Dynkin formula) . 
Furthermore, one can classify solutions of the homogeneous
equation $[x,a]+[y,b]=0$ using Lemma in Section 6, \cite{Dr} (see below),
or using the technique of  \cite{Burgunder}.

The full KV problem admits a solution using the
Kontsevich deformation quantization technique \cite{AM},
and there is another solution \cite{AT} using the Drinfeld's theory of 
associators \cite{Dr}. At the same time, it is known that the
quadratic KV problem is much easier. In particular, it admits
explicit rational solutions \cite{V}, \cite{AMquad} (whereas in the general
case, it is plausible that an explicit rational solution of the KV
problem amounts to finding an explicit rational associator).
There are also two elementary proofs of the Duflo theorem
for quadratic Lie algebras: one using Clifford calculus \cite{AMduf},
and one using the Kontsevich integral in knot theory \cite{knots}. 

These simplifications in the quadratic case
are explained by the following theorem:

\begin{theorem}  \label{THM}
Every solution of equation \eqref{kv1} verifies equation
\eqref{kvquad2}.
\end{theorem}

This observation shows that the quadratic KV problem
reduces to equation \eqref{kv1}.
In particular, this explains why rational solutions of the quadratic
KV problem are easy to obtain: every rational factorization of 
the Campbell-Hausorff
series of the form $\ch(y,x)=x+y-[x,a(x,y)]-[y,b(x,y)]$ 
gives rise to such a solution.
It also explains why easy proofs of the Duflo theorem for quadratic
Lie algebras are available: in this case, the Duflo theorem follows
from the fact that the Campbell-Hausdorff series factorizes 
according to equation \eqref{kv1}.

To prepare the proof of Theorem \ref{THM}, we recall  Lemma
in Section 6, \cite{Dr}. Let $\Tr_n^2$ be the linear span of 
expressions of the form  $\tr(ab)$ for $a,b\in \Lie_n$. 
The lemma states that there is a one-to-one correspondence between 
elements $p \in \Tr_n^2$ and $n$-tuples $a_1, \dots, a_n \in \Lie_n$
satisfying $\sum_{i=1}^n [x_i, a_i]=0$. This correspondence
is given by formula  
$\frac{d}{ds} \, p(x_1, \dots, x_i+sz, \dots, x_n) |_{s=0} =
\tr(z a_i)$.

\begin{proposition}   \label{key}
Let $a_1, \dots, a_n \in \Lie_n$ such that 
$\sum_{i=1}^n [x_i, a_i]=0$, and let $u \in \tder_n$ be the tangential
derivation defined by the $n$-tuple $(a_1, \dots, a_n)$.
Then,
$$
\div^{quad}(u)=\sum_{i=1}^n\tr^{quad}(x_i(\partial_i a_i)) =0
$$  
\end{proposition}

\begin{proof}
Let $p \in \Tr_2^2$ be an element generating the derivation $u$.
Consider
$$
\begin{array}{lll}
\frac{\partial^2}{\partial s \partial t} \,
p(x_1, \dots, x_i+sz_1+tz_2, \dots x_n)|_{s=t=0} & = &
\frac{\partial}{\partial t} \, 
\tr(z_1 a_i(x_1,\dots,x_i+tz_2,\dots,x_n))|_{t=0} \\
& = & \tr(z_1 ({\rm ad}(\partial_i a) z_2)) \\
& = & \tr(({\rm ad}(\tau(\partial_i a)) z_1) z_2) .
\end{array}
$$
Since the left hand side is symmetric in $z_1, z_2$, we conclude
that $\partial_i a_i = \tau(\partial_i a_i)$. Then,
$$
\begin{array}{lll}
\div^{quad}(u) & = & \sum_{i=1}^n \tr^{quad}(x_i (\partial_i a_i)) \\
& = & \sum_{i=1}^n \tr^{quad}(\tau(x_i (\partial_i a_i))) \\
& = & - \sum_{i=1}^n \tr^{quad} ((\partial_i a_i) x_i) \\
& = & - \div^{quad}(u),
\end{array}
$$
and $\div^{quad}(u)=0$, as required.
\end{proof}

\begin{remark}

Here we presented an algebraic proof of Proposition \ref{key}
suggested to us by Michele Vergne.
One can also give a proof using graphical calculus as in Section 5.1,
\cite{SW}.
\end{remark}

The next proposition summarizes known properties of equation
\eqref{kv1}.

\begin{proposition} \label{prop:kv1}
The following three statements are equivalent:
\begin{itemize}
\item
$A,B \in \Lie_2$ is a solution of equation \eqref{kv1}.

\item 
For all $t \in \k^*$, the tangential derivation $u_t \in \tder_2$
defined by formula $u_t(x)=t^{-1}[x,A(tx,ty)], u_t(y)=t^{-1}[y,B(tx,ty)]$
verifies  equation
$$
u_t(\ch_t(x,y))= \frac{d}{dt} \, \ch_t(x,y).
$$

\item
The solution $F_t \in \TAut_2$ of the differential equation 
$$
F_t^{-1} \, \frac{dF_t}{dt} \, =u_t
$$
with initial condition $F_0=1$ verifies equation $F_t(\ch_t(x,y))=x+y$.
\end{itemize}
\end{proposition}

\begin{proof}
For equivalence of the first and second statements, see
Lemma 3.2, \cite{KV}. Equivalence of the second and  third
statements is obvious, see also Theorem 5.2, \cite{AT}.
\end{proof}

Note that given $F \in \TAut_2$ verifying $F(\ch(x,y))=x+y$ (as in
Proposition \ref{prop:kv1}), one can construct $\F \in \TAut_n$
by formula $\F=F^{1,2}F^{12,3} \dots F^{1\dots (n-1),n}$ such that
$$
\begin{array}{lll}
\F (\ch(x_1,\dots,x_n)) & = &
F^{1,2}F^{12,3} \dots F^{1\dots (n-1),n}(x_1+\dots+x_n) \\
& = & F^{1,2}F^{12,3} \dots F^{1\dots (n-2),n-1} (\ch(x_1,\dots,x_{n-1})+x_n) \\
& = & \dots \\
& = & F^{1,2} (\ch(x_1,x_2)+x_3+ \dots +x_n) \\
& = & x_1 + \dots + x_n.
\end{array}
$$

\begin{proposition} \label{last}
Let $u \in \tder_n$ such that $u(\ch(x_1,\dots,x_n))=0$. Then, $\div^{quad}(u)=0$.
\end{proposition}

\begin{proof}
Define $v = \Ad_{\F} u \in \tder_n$.
We have 
$$
v(x_1+\dots+x_n)=\F(u(\F^{-1}(x_1+\dots+x_n)))=
\F(u(\ch(x_1,\dots,x_n)))=0.
$$
Then,
$$
\div^{quad}(u)=\div^{quad}(\Ad_{\F^{-1}}(v))=\F^{-1}\cdot \div^{quad}(v)=0,
$$
where we used Proposition \ref{key} and 
the cocycle property of $\div^{quad}$.
\end{proof}

\begin{proposition}  \label{U}
Let $A,B \in \Lie_2$ be a solution of equation \eqref{kv1}, and $u\in \tder_2$
be the corresponding tangential derivation. Then, 
\begin{equation} \label{eq:U}
U=u^{1,2}+u^{12,3}-u^{1,23}-u^{2,3}
\end{equation}
verifies $U(\ch(x,y,z))=0$.
\end{proposition}

\begin{proof}
Consider $\ch_t(x,y,z)=\ch_t(\ch_t(x,y),z)$. We have,
$$
\begin{array}{lll}
\frac{d}{dt} \, \ch_t(x,y,z)|_{t=1} & = &
(\frac{d}{dp}\, \ch_p(\ch_q(x,y),z) + \frac{d}{dq}\, \ch_p(\ch_q(x,y),z))_{p=q=1} \\
& = & u^{12,3}(\ch(x,y,z)) + u^{1,2}(\ch(x,y,z)).
\end{array}
$$
Similarly, we obtain
$$
\begin{array}{lll}
\frac{d}{dt} \, \ch_t(x,y,z)|_{t=1} & = &
(\frac{d}{dp}\, \ch_p(x,\ch_q(y,z)) + \frac{d}{dq}\, \ch_p(x,\ch_q(y,z)))_{p=q=1} \\
& = & u^{1,23}(\ch(x,y,z)) + u^{2,3}(\ch(x,y,z)).
\end{array}
$$
By combining these two equations we arrive at $U(\ch(x,y,z))=0$.
\end{proof}

\begin{proposition}
Let $A,B \in \Lie_2$ be a solution of equation \eqref{kv1}, $u\in \tder_2$
be the corresponding tangential derivation, and $g=\div^{quad}(u) \in \tr_2^{quad}$.
Then, there is $h \in \tr_1^{quad}$ such that
$g(x,y)=h(x)+h(y)-h(\ch(x,y))$.
\end{proposition}

The main part of the proof is borrowed from \cite{AT}, Theorem 2.1. We reproduce
it here for convenience of the reader.

\begin{proof}
Propositions \ref{last} and \ref{U} imply that  for $U\in \tder_3$ 
defined by equation \eqref{eq:U} we have $\div^{quad}(U)=0$. That is,
\begin{equation}
\begin{array}{lll}  \label{inhomo}
0 = \div^{quad}(U) & = & 
\div^{quad}(u^{1,2})+ \div^{quad}(u^{\tilde{12},3}) 
- \div^{quad}(u^{1,\tilde{23}})- \div^{quad}(u^{2,3}) \\
& = & g(x,y) + g(\ch(x,y),z) - g(x,\ch(y,z))-g(y,z) .
\end{array}
\end{equation}

Consider an auxiliary equation
with the Campbell-Hausdorff series $\ch(x,y)$ replaced by $x+y$.
In more detail, we look for all $g(x,y) \in \tr^{quad}_2$ which verify
\begin{equation} \label{homo}
g(x,y) + g(x+y,z) - g(x,y+z)-g(y,z)=0.
\end{equation} 
In order to solve this equation, we first put $x \mapsto sx, y \mapsto x, z \mapsto z$
to get
$$
g(sx,x)+g((1+s)x,z)-g(sx,x+z)-g(x,z)=0.
$$
Similarly, by putting $x\mapsto x, y \mapsto z, z \mapsto sz$ we obtain
$$
g(x,z)+g(x+z,sz)-g(x,(1+s)z)-g(z,sz)=0.
$$
Since equation \eqref{homo} preserves the degree, we can assume without
loss of generality that $g$ is homogeneous of degree $n$. Then, we have
$$
\begin{array}{lll}
n g(x,z) & = & 
\frac{d}{ds} \, (g((1+s)x,z) + g(x,(1+s)z)) |_{s=0} \\
& = & \frac{d}{ds} \, (g(sx,x+z)+g(x+z,sz)-g(sx,x)-g(z,sz))|_{s=0} .
\end{array}
$$
That is, we get 
$g(x,z)=\tr^{quad}((\alpha x + \beta z)(x+z)^{n-1} - \alpha x^n - \beta z^n)$.
This expression vanishes for $n$ odd because $\tr^{quad}(x^n)=\tr^{quad}(\tau(x^n))=
(-1)^n \tr^{quad}(x^n)$ (and by a similar argument for other terms).
For $n$ even, we compute
$$
\begin{array}{ll}
& g(x,y) + g(x+y,z) - g(x,y+z)-g(y,z) \\
= & (\beta - \alpha)\tr^{quad} y((x+y)^{n-1}+(y+z)^{n-1}-(x+y+z)^{n-1}-y^{n-1}).
\end{array}
$$
In particular, the coefficient in front of $\tr^{quad}(y^{n-2}(xz+zx))$ 
is equal to $(\beta-\alpha)(n-2)$ which implies $\beta=\alpha$ for $n\neq 2$.
That is, for $n\neq 2$ we get
$$
g(x,z)=\alpha \tr^{quad}((x+z)^n-x^n-z^n).
$$
Furthermore, for $n=2$ we obtain
$$
g(x,z)=(\alpha+\beta) \tr^{quad}(xz)= \frac{\alpha+\beta}{2}\, \tr^{quad}((x+z)^2-x^2-z^2).
$$
In summary, all solutions of equation \eqref{homo} are of the form
$g(x,z)=h(x)+h(z)-h(x+z)$ for $h \in \tr^{quad}_1$.

Getting back to equation \eqref{inhomo}, let $g=\sum_{k \geq n} g_k$ be a solution,
where $g_k$ are homogeneous components of degree $k$. By taking the degree $n$ part
of equation \eqref{inhomo} we recover equation \eqref{homo} for $g_n$. Hence,
there is $h_n \in \tr^{quad}_1$ such that $g_n(x,y)=h_n(x)+h_n(y)-h_n(x+y)$.
Consider $\tilde{g}(x,y)=g(x,y)-(h_n(x)+h_n(y)-h_n(\ch(x,y))$. It is easy to see
that $\tilde{g}$ still verifies equation \eqref{inhomo}, and that it starts in degree
$n+1$. Proceeding by induction, we show that all solutions of equation \eqref{inhomo}
are of the form $g(x,y)=h(x)+h(y)-h(\ch(x,y))$ for some $h\in \tr_1^{quad}$, 
as required.
\end{proof}

We conclude the proof of the main result of this paper with the following
Proposition.

\begin{proposition}
Let $A,B \in \Lie_2$ be a solution of \eqref{kv1}, 
$u\in \tder_2$ be the corresponding
tangential derivation, and assume $\div^{quad}(u)=h(x)+h(y)-h(\ch(x,y))$ for 
some $h \in \tr_1^{quad}$. Then, $h(x)=\tr^{quad} f(x)$ for $f=x/(e^x-1)-1+x/2$.
\end{proposition}

In the proof, we follow the ideas \cite{AP} (see Remark 4.3). The main part
of the argument is the same as in the proof of Proposition 6.1 in \cite{AT}.
We reproduce it for convenience of the reader.

\begin{proof}
Write $A(x,y)=ax+\alpha(\ad_x) y + \dots, B(x,y)=bx + \beta(\ad_x)y + \dots$,
where $a,b \in \k, \alpha, \beta \in \k[[x]]$, and $\dots$ stand for 
terms containing at least two $y$'s.
Replace $y \mapsto sy$ in equation \eqref{kv1}, and compute the 
first and second derivatives in $s$ at $s=0$. The first derivative yields
$$
y - \, \frac{\ad_x}{e^{\ad_x} -1 } \, y = (1-e^{-\ad_x}) \alpha(\ad_x) y
- b [x,y],
$$
and we obtain
$$
\alpha(t)=b \, \frac{t}{1-e^{-t}} \, - \, \frac{t}{(e^t-1)(1-e^{-t})} + \, \frac{1}{1-e^{-t}} .
$$
Note that elements of degree two in $y$ of $\Lie_2$ are 
in bijection with skew-symmetric formal power series in two variables,
$$ 
a(u,v)=\sum_{i,j=0}^\infty a_{i,j} u^i v^j \mapsto
\sum_{i,j=0}^\infty a_{i,j} [\ad_x^i y, \ad_x^j y]
$$
The second derivative of \eqref{kv1} gives the following equality 
in formal power series,
$$
\frac{1}{2} \, \frac{(u+v)(e^u-e^v) - (u-v)(e^{u+v}-1)}{(e^{u+v}-1)(e^u-1)(e^v-1)} =
(1-e^{-(u+v)}) a_2(u,v) + \frac{b}{2} (u-v) + (\beta(v)-\beta(u)),
$$
where the left hand side corresponds to the second derivative of the
Campbell-Hausdorff series $-\ch(sy,x)$, and
$a_2(u,v)$ represents the second derivative of $A(x,sy)$. By putting
$v=-u$ in the last equation we obtain,
$$
\beta_{odd}(t)= \frac{b}{2} \, t - \frac{1}{2} \, \frac{t}{(e^t-1)(1-e^{-t})}
+ \frac{1}{4} \, \frac{e^t+1}{e^t-1} \, .
$$
Here $\beta_{odd}(t)=(\beta(t)-\beta(-t))/2$.

Finally, consider equation
$$
\tr^{quad}(x (\partial_x A) + y (\partial_y B)) = \frac{1}{2}\, 
\tr^{quad}(f(x)+f(y) - f(\ch(x,y)) ),
$$
and compute the contribution linear in $y$ (that is, of the form  
$\tr^{quad}(a(x) y)$) on the left hand side and on the right hand side.
Note that $\tr^{quad}(x^n)$ and $\tr^{quad}(x^{n-1}y)$ vanish for $n$ odd.
Hence, it is sufficient to look at even degrees. In this way, we obtain
$$
\beta_{odd}(t)-\alpha_{odd}(t)=-\frac{1}{2} \, \frac{df}{dt}
$$
which implies
$$
f(t)= \frac{t}{e^t-1} -1 + \frac{t}{2},
$$
as required.

\end{proof}

\vskip 0.2cm

{\bf Acknowledgements.} We are grateful to M. Duflo,
M. Kashiwara, E. Meinrenken,
P. Severa, M. Vergne and T. Willwacher for useful discussions.
Research of A.A. was supported in part by the grants of the 
Swiss National Science Foundation number 200020-120042 and
number 200020-121675. Research of C.T. was supported by CNRS.

\end{document}